\documentclass[12pt]{article}
\usepackage{graphicx}        
\usepackage{multicol}        
\usepackage[bottom]{footmisc}
\usepackage{subcaption}
\usepackage{a4wide,amsmath,amssymb,amsthm,xcolor,booktabs}

\newtheorem{theorem}{Theorem}[section]

\newtheorem{lemma}[theorem]{Lemma}
\newtheorem{remark}{Remark}[section]

\newcommand{\R}{\mathbb{R}}

\newcommand{\norm}[1]{{\left\lVert{#1}\right\rVert}}
\newcommand{\skpr}[1]{\langle #1 \rangle}

\newcommand{\OC}{\operatorname{\mathsf{C}}}
\newcommand{\OD}{\operatorname{\mathsf{D}}}
\newcommand{\OS}{\operatorname{\mathsf{S}}}
\newcommand{\OT}{\operatorname{\mathsf{T}}}
\newcommand{\OcT}{\operatorname{\mathcal{T}}}

\begin{document}
\title{Paving the way to a $\OT$-coercive method\\
for the wave equation}
\author{Daniel Hoonhout$^\ast$, Richard L\"oscher$^\dagger$,  Carolina Urz\'{u}a-Torres$^\ast$\\ 
$^\ast$
  Delft Institute of Applied Mathematics, TU Delft,\\
   Mekelweg 4, 2628CD Delft,
  The Netherlands \\ d.m.hoonhout@tudelft.nl, c.a.urzuatorres@tudelft.nl\\
  $^\dagger$Institute of Applied Mathematics, TU Graz,\\
 Steyrergasse 30, 8010 Graz,
  Austria \\ loescher@math.tugraz.at\\
}
\maketitle

\abstract{
In this paper, we take a first step toward introducing a space-time transformation operator $\OT$ that establishes $\OT$-coercivity for the weak variational formulation of the wave equation in space and time on bounded Lipschitz domains. As a model problem, we study the ordinary differential equation (ODE)
$u'' + \mu u = f$ for $\mu>0$, 
which is linked to the wave equation via a Fourier expansion in space. For its weak formulation, we introduce a transformation operator $\OT_\mu$ that establishes 
$\OT_\mu$-coercivity of the bilinear form yielding an 
unconditionally stable Galerkin-Bubnov formulation with error estimates independent of $\mu$.
The novelty of the current approach is the explicit dependence of the transformation on 
$\mu$ which, when extended to the framework of partial differential equations, yields an operator acting in both time and space. We pay particular attention to keeping the trial space as a standard Sobolev space, simplifying the error analysis, while only the test space is modified. The theoretical results are complemented by numerical examples.  
}

\section{Introduction}
There has been an increased interest in finding space-time formulations for 
the wave equation that lead to unconditionally stable discretizations. Some of 
them rely on analyzing the exact mapping properties of the space-time wave 
operator in a variational setting 
\cite{FuehrerGonzalesKarkulik:2025,HLSU:2023,KoetheLoescherSteinbach:2023,SteinbachZank:2020}, 
including weak and least squares formulations. These often lead to a setting 
beyond standard Sobolev spaces, that is cumbersome to discretize. Others though, 
stay in the framework of Sobolev spaces, adding stabilization terms 
\cite{FraschiniLoliMoiolaSangalli:2024,SteinbachZank:2020a}, or applying special 
test functions \cite{BignardiMoiola:2025,FerrariPerugiaZampa:2025,SteinbachZank:2020}. 
The overall goal of all these approaches is to open the door to adaptivity in 
space and time simultaneously and without additional computational cost. 

In this paper, we follow the idea from \cite{SteinbachZank:2020} and consider a 
second order ordinary differential equation (ODE) $u''+\mu u =f$ with $\mu>0$. 
This ODE can be linked to the wave equation and used to find a transformation 
operator such that one obtains a coercive and continuous bilinear form for the 
corresponding variational formulation. However, unlike previous approaches, we 
take pains to make sure the resulting transformation also depends on $\mu$, and 
thus, when extended to the partial differential equation framework, can lead to 
a transformation that not only acts on time, but also on space. Moreover, we are 
interested in keeping standard Sobolev spaces as our solution spaces, such that 
the error analysis of future space-time methods and understanding of the solution 
is straightforward. With this purpose in mind, we deliberately insist on using 
the arising transformation to change the test space. Interestingly, this results 
in a variational formulation with a simple bilinear form, while only the 
right-hand side becomes more involved.

This contribution focuses on the key ideas for finding this transformation and 
therefore restricts attention to the ODE case. Ongoing and future work deals 
with the extension to the actual wave equation.

\section{Motivation}
Let $T>0$, $\mu>0$ and $f:(0,T)\to \R$ be given and let us consider the ordinary
differential equation to find $u:[0,T]\to\R$ such that 
\begin{equation}\label{eq:ODE-wave}
    \begin{aligned}
        u''(t) + \mu u(t) &= f(t)\quad \text{ for }t\in (0,T),\qquad
        u(0)=u'(0)=0. 
    \end{aligned}
\end{equation}
For the weak form, we multiply by 
a smooth function $v$ and integrate by parts to get
\begin{align*}
    \int_0^T \big[u''(t)v(t)+\mu u(t)v(t)\big]\, dt 
    & = \int_0^T \big[-u'(t)v'(t) + \mu u(t)v(t)\big]\, dt +u'(T)v(T).
\end{align*}
Now, in order to get rid of the term $u'(T)$, about which we have no information, 
and to ensure well-defined integrals, we need the spaces
\begin{align*}
H^1_{0,}(0,T) :=\{ v \in H^1(0,T)|\, v(0)=0\},\qquad
H^1_{,0}(0,T) :=\{ v \in H^1(0,T)|\, v(T)=0\}.
\end{align*}

With these, we consider the variational formulation: For $f\in [H^1_{,0}(0,T)]'$, 
find $u\in H^1_{0,}(0,T)$ such that for all $v\in H_{,0}^1(0,T)$
\begin{equation}\label{eq:VF-ODE-wave}
    b_\mu(u,v) := \int_0^T\big[-u'(t)v'(t)+\mu u(t)v(t)\big]\, dt = \int_0^T 
    f(t)v(t)\, dt. 
\end{equation}  
\begin{theorem}{{\cite[Lemma 4.2.1, Lemmata 4.2.3--4.2.4]{Zank:2020}}}\label{thm:H1-setting}
$\:$ The bilinear form\\ $b_\mu:H_{0,}^1(0,T)\times H_{,0}^1(0,T)\to \R$ defined 
in \eqref{eq:VF-ODE-wave} satisfies 
\begin{description}
    \item[(B1)] Boundedness: For all $u \in H^1_{0,}(0,T)$ and $v\in H_{,0
    }^1(0,T)$ it holds that
    \begin{equation*}
        |b_\mu(u,v)|\leq \left(1+\frac{4T^2\mu}{\pi^2}\right)\norm{u'}_{L^2
        (0,T)}\norm{v'}_{L^2(0,T)}. 
    \end{equation*} 
    \item[(B2)] Bounded invertibility: For all $u\in H_{0,}^1(0,T)$ it holds 
    that 
    \begin{equation*}
        \frac{2}{2+T\sqrt{\mu}}\norm{u'}_{(0,T)}\leq \sup_{0\neq v\in H_{
        ,0}^1(0,T)}\frac{b_\mu(u,v)}{\norm{v'}_{(0,T)}}. 
    \end{equation*}
    \item[(B3)] Surjectivity: For all $0\neq v\in H_{,0}^1(0,T)$ there exists 
    $u_v\in H_{0,}^1(0,T)$ such that $ b_\mu(u_v,v) \neq 0. $
\end{description}
\end{theorem}
The above theorem guarantees that the weak formulation 
\eqref{eq:VF-ODE-wave} admits a unique solution $u\in H_{0,}^1(0,T)$ for all 
right-hand sides $f\in[H_{,0}^1(0,T)]'$ and fixed $\mu>0$. But, the obtained 
bounds depend on $\mu>0$ and, in particular, they degenerate 
when $\mu\to\infty$. 

Our main contribution is to propose a formulation where we keep the 
solution $u\in H^1_{0,}(0,T)$ in the standard Sobolev space but 
get $\mu$-independent bounds. For this, we introduce
\begin{equation*}
    \norm{w}_{H^1_\mu} := \sqrt{\norm{w'}_{(0,T)}^2+\mu\norm{w}_{(0,T)}^2},
\end{equation*}
which resembles the space-time $H^1$-norm, and defines an equivalent norm on 
$H^1_{0,}(0,T)$ and $H^1_{,0}(0,T)$. In what follows, we will construct a 
transformation operator $\OT_\mu:H^1_{0,}(0,T)\to H^1_{,0}(0,T)$ such that the 
bilinear form satisfies
\begin{equation*}
    \norm{u}_{H^1_\mu}^2 = b_\mu(u,\OT_\mu u). 
\end{equation*} 

\section{A space-time transformation!? Well, it's complex...}
In the following we slightly abuse notation and use $\skpr{\cdot,\cdot}$ 
to denote both the $L^2$-inner product and the $L^2$-duality pairing. 
Now, using integration by parts, we get
\begin{equation}\label{eq:ibp}
 \skpr{u',v} + \skpr{u,v'}=0, \qquad \forall 
 u\in H^1_{0,}(0,T), \forall v\in H^1_{,0}(0,T).
\end{equation}
We introduce the operators 
\begin{equation*}
    \OD^\pm := \pm i\partial_t +\sqrt{\mu}, 
\end{equation*}
where $i$ denotes the imaginary unit. Using \eqref{eq:ibp}, we have that 
\begin{equation*}
\skpr{\OD^+ u,\OD^+v} = -\skpr{u',v'} + \mu\skpr{u,v} + i\sqrt{\mu} 
(\skpr{u',v} + \skpr{u,v'}) = b_\mu(u,v). 
\end{equation*}
Moreover, note that for any real valued function $u\in H^1_{0,}(0,T)$
\begin{align*}
\skpr{\OD^+ u,\OD^- u} &= \skpr{\OD^+ u ,\overline{\OD^+ u}} = \norm{\Re(\OD^+ u)}_{L^2(0,T)}^2+\norm{\Im(\OD^+ u)}_{L^2(0,T)}^2\\
&= \mu\norm{u}_{L^2(0,T)}^2+\norm{u'}_{L^2(0,T)}^2 = \norm{u}_{H^1_\mu}^2. 
\end{align*}
Here, for any function $w$, $\Re(w)$ and $\Im(w)$ denote its real and 
imaginary parts, respectively, while $\overline{w}$ corresponds to its 
complex conjugate.

The above motivates us to introduce the transformation operator 
\begin{equation}\label{eq:transformation1}
\OcT_\mu:=\left(\OD^{+}\right)^{-1}\OD^-,
\end{equation}
since for $v = \OcT_\mu u$ we then formally have that
\begin{equation*}
    b_\mu(u,v) = \skpr{\OD^+ u,\OD^+\OcT_\mu u} = \skpr{\OD^+ u,
    \OD^- u} = \norm{u}_{H^1_\mu}^2. 
\end{equation*}

\subsection{The transformation operator: Its actual closed form formula}
We proceed to compute the transformation operator defined in 
\eqref{eq:transformation1} acting on a function in $H^1_{0,}(0,T)$.
Moreover, in order to stay consistent, we would like that the operator 
$\OcT_\mu$ maps functions from $H^1_{0,}(0,T)$ to functions with zero terminal 
condition. 

As the solution of 
\begin{equation*}
    \OD^+ z := iz'+\sqrt{\mu}z = q,\, \text{in }(0,T),\quad z(T) = 0,
\end{equation*}
can be explicitly computed to be
\begin{equation*}
    z(t) = \left(\left(\OD^{+}\right)^{-1} q\right)(t) = i\int_t^T e^{i\sqrt{
    \mu}(t-s)}q(s)\, ds,
\end{equation*}
we have that for any $w\in H^1_{0,}(0,T)$
\begin{align}
\label{eq:complexForm}
    v(t) &= (\OcT_\mu w)(t) = \left(\left(\OD^+\right)^{-1}\OD^-w\right)(t)
    \nonumber\\ 
    &= i\int_t^T e^{i\sqrt{\mu}(t-s)}(-iw'(s)+\sqrt{\mu}w(s))\, ds.
\end{align}

\subsection{A convenient REALization}

We are only interested in testing with real valued functions $v$. However, for 
any real valued function $w\in H^1_{0,}(0,T)$, we see from \eqref{eq:complexForm} 
that $\OcT_\mu w$ is complex valued. 
To circumvent the use of complex functions, we now consider only the real part 
of the transformation operator and define for $w\in H^1_{0,}(0,T)$
\begin{equation*}
\OT_\mu w(t) \hspace{-0.2em}:= \Re(\OcT_\mu \hspace{-0.2em} w(t)) = \int_t^T\hspace{-0.2em}\left[\cos(\sqrt{\mu}(t-s))w'(s) 
-\sqrt{\mu}\sin(\sqrt{\mu}(t-s))w(s)\right]\,\hspace{-0.2em} ds.
\end{equation*}
\begin{lemma}\label{lem:properties of T_mu real}
The operator $\OT_\mu:H^1_{0,}(0,T)\to H^1_{,0}(0,T)$ is well-defined, bounded and
satisfies for all $u\in H^1_{0,}(0,T)$ and all $q\in L^2(0,T)$
\begin{equation*}
    \int_0^T \left[-(\OT_\mu u)'(t) +\mu \hskip -0.4em \int_t^T \hskip -0.4em 
    (\OT_\mu u)(s)\, ds\right] q(t)\, dt = \int_0^T \left[u'(t) + \mu \hskip 
    -0.4em \int_t^T \hskip -0.4em u(s)\, ds\right] q(t)\, dt. 
\end{equation*}
\end{lemma}
\begin{proof}
By construction $(\OT_\mu u)(T) =0$ and we compute
\begin{equation*}
(\OT_\mu u)'(t) = -u'(t) - \sqrt{\mu}\int_t^T \left[\sin(\sqrt{\mu}(t-s))u'(s) 
+ \sqrt{\mu}\cos(\sqrt{\mu}(t-s))u(s)\right]\, ds. 
\end{equation*}
For fixed $\mu$, all the terms on the right-hand side are bounded in $L^2(0,T)$ 
if $u\in H^1_{0,}(0,T)$, and thus $\OT_\mu$ is well defined. Using a triangle and 
Cauchy--Schwarz inequalities, we also get the bound
\begin{equation}\label{eq:bound}
 \norm{\OT_\mu u}_{H^1_{,0}(0,T)}\leq \left(1+T\sqrt{\frac{\mu}{2}}\right)\norm{u}_{H^{1}_\mu}, 
 \quad u\in H^1_{0,}(0,T).
\end{equation}
Furthermore, using that $(\OT_\mu u)'(T)
= -u'(T)$, we compute that 
\begin{equation*}
    (\OT_\mu u)'' + \mu \OT_\mu u = -u'' + \mu u \, \text{ in }[H^1_{0,}(0,T)]'. 
\end{equation*}
Hence, for all $w\in H^1_{0,}(0,T)$ we have
\begin{equation*}
-\skpr{(\OT_\mu u)',w'} + \mu \skpr{\OT_\mu u,w} = \skpr{u',
w'}+\mu \skpr{u,w}, 
\end{equation*}
where we applied the integration by parts formula \eqref{eq:ibp}.
Next, we use that each $w\in H^1_{0,}(0,T)$ admits the representation $w(t) = 
\int_0^t q(s)\, ds$ for some $q\in L^2(0,T)$. Thus, $w'=q$ and using the 
following integration by parts formula 
\begin{equation}\label{eq:integration by parts for integrals}
\int_0^T p(t) \int_0^t q(s)\, ds \, dt = \int_0^T \int_t^T p(s)\, ds \, q(t)\, 
dt,\quad \forall p,q\in L^2(0,T), 
\end{equation}
concludes the proof. 
\end{proof}
\begin{remark}
Note that, for $\mu=0$, $\OT_0 w(t) = w(T) - w(t) = \overline{\mathcal{H}
}_Tw(t)$, which was introduced in \cite[Lemma 4.5]{SteinbachZank:2020} as a 
purely temporal transformation for the stabilization of the wave equation. Hence, 
$\OT_\mu$ can be seen as an extension of $\overline{\mathcal{H}}_T$ to 
space-time. 
\end{remark}

\vspace{-2em}
\section{Always Look on the Right Side of Life}
Using the real transformation $\OT_\mu$, we are now considering the variational 
formulation: Given $f\in[H^1_{,0}(0,T)]'$, find $u\in H^1_{0,}(0,T)$ such that 
\begin{equation}\label{eq:VF transformation}
b_\mu(u,\OT_\mu w) = \skpr{f,\OT_\mu w},\quad \forall w\in H^1_{0,}(0,T).
\end{equation}
We first prove that its application on the second argument of $b_\mu$ does not 
introduce complicated expressions or additional terms.
\begin{lemma}\label{lem:properties bmu}
For all $u,w\in H^1_{0,}(0,T)$ it holds that
\begin{equation*}
    b_\mu(u,\OT_\mu w) = \skpr{u',w'} + \mu \skpr{u,w}.
\end{equation*}
In particular, $b_\mu$ is bounded and $\OT_\mu$-coercive, i.e., 
\begin{equation*}
    b_\mu(u,\OT_\mu w)\leq \norm{u}_{H^1_\mu}\norm{w}_{H^1_\mu},\quad 
    \text{and}\quad b_\mu(u,\OT_\mu u) = \norm{u}_{H^1_\mu}^2. 
    \end{equation*}
\end{lemma}
\begin{proof}
As $u\in H^1_{0,}(0,T)$ there exists $q\in L^2(0,T)$ such that $u(t) = \int_0^t 
q(s)\, ds.$ Using \eqref{eq:integration by parts for integrals} 
and Lemma \ref{lem:properties of T_mu real}, we then compute for all $w \in 
H^1_{0,}(0,T)$
\begin{align*}
    b_\mu(u,\OT_\mu w)& = \int_0^T \left[-u'(t)(\OT_\mu w)'(t) + \mu u(t)(\OT_\mu w)(t)\right]\, dt\\
    & = \int_0^T q(t)\left(-(\OT_\mu w)'(t) + \mu \int_t^T(\OT_\mu w)(s)\, ds\right)\, dt\\
    & = \int_0^T q(t)\left(w'(t) + \mu \int_t^T w(s)\, ds \right)\, dt \\
    & = \int_0^T u'(t)w'(t) + \mu u(t)w(t)\, dt. 
\end{align*}
\end{proof}
\noindent The main result is now the following. 
\begin{theorem}
The variational formulation \eqref{eq:VF transformation} admits a unique solution 
$u\in H^1_{0,}(0,T)$ for all $f\in[H^1_{,0}(0,T)]'$. In particular, for all
\vskip -1em 
\begin{equation*}
f\in\mathcal{F} := \{f\in [H^1_{,0}(0,T)]': \skpr{f,\OT_\mu w}\leq C\norm{w}_{H^1_\mu},\, \forall w\in H^1_{0,}(0,T)\}.
\end{equation*} 
\vskip -0.5em 
\noindent
with some $C>0$ independent of $\mu$, the solution $u$ satisfies the $\mu$-independent bound
\begin{equation}\label{eq:bound of solution f in F}
    \norm{u}_{H^1_\mu} \leq C. 
\end{equation}
\end{theorem}
\begin{proof}
By Lemma \ref{lem:properties bmu} the bilinear form $b_\mu$ is $\OT_\mu$-coercive 
and bounded. Moreover, for fixed $\mu>0$, using \eqref{eq:bound} we get 
\begin{equation*}
    \skpr{f,\OT_\mu w}\leq \norm{f}_{[H^1_{,0}(0,T)]'}\norm{\OT_\mu w}_{
    H^1_{,0}(0,T)} \leq \left(1+T\sqrt{\frac{\mu}{2}}\right)\norm{f}_{[H^1_{,0}(0,T)]'}\norm{w}_{H^1_\mu}. 
\end{equation*}
Hence, by the Lemma of Lax-Milgram, there exists a unique solution $u\in H^1_{0,}
(0,T)$ for all $f\in [H^1_{,0}(0,T)]'$.
The bound \eqref{eq:bound of solution f in F} for $f\in\mathcal{F}$ now follows from 
\begin{equation*}
\norm{u}_{H^1_\mu}^2 = b_\mu(u,\OT_\mu u) = \skpr{f,\OT_\mu u}\leq C
\norm{u}_{H^1_\mu}. 
\end{equation*}
\end{proof}
\begin{remark}
For $w\in H^1_{0,}(0,T)$, using \eqref{eq:integration by parts for integrals} and the Cauchy--Schwarz inequality, we get
\begin{align*}
    \skpr{f,\OT_\mu w} &= \int_0^T \left[w'(t)\OC(t,f)+ \sqrt{\mu}w(t)\OS(t,f)\right]\, dt\\
     &\leq \norm{w}_{H^1_\mu}\left(\int_0^T \left[\OC(t,f)^2+\OS(t,f)^2\right]\, dt\right)^{1/2}
\end{align*}
where $\OC(t,f) = \int_0^t \cos(\sqrt{\mu}(t-s))f(s)\, ds$ and $\OS(t,f) = \int_0^t \sin(\sqrt{\mu}(t-s))f(s)\, ds$.
For $f\in L^2(0,T)$ we can now proceed as in \cite[Lemma 4.9]{SteinbachZank:2020}, to compute 
\begin{equation*}
    \int_0^T \left[\OC(t,f)^2+\OS(t,f)^2\right]\, dt \leq \frac{T^2}{2}\norm{f}_{(0,T)}^2,
\end{equation*}
showing that $L^2(0,T) \subset \mathcal{F}$ and \eqref{eq:bound of solution f in F} 
holds with $C = \frac{T}{\sqrt{2}}\norm{f}_{(0,T)}$. 
\end{remark}
\vspace{-2em}
\section{Discretization}
We consider the conforming trial space $X_h := S_{h}^1(0,T)\cap H^1_{0,}(0,T)$ 
of piecewise linear finite elements  defined on a uniform decomposition of the 
interval $(0,T)$ into $N\in\mathbb{N}$ elements of mesh size $h=T/N$. The 
discrete variational formulation reads: Given $f\in[H^1_{,0}(0,T)]'$, find $u_h\in 
X_h$ such that 
\begin{equation}\label{eq:DVF transformation}
b_\mu(u_h,\OT_\mu w_h) = \skpr{u_h',w_h'} + \mu \skpr{u_h,w_h} 
= \skpr{f,\OT_\mu w_h},\, \forall w_h\in X_h. 
\end{equation}
Due to the $\OT_\mu$-coercivity, \eqref{eq:DVF transformation} admits a 
unique solution $u_h\in X_h$. By Lemma \ref{lem:properties bmu} and 
the linearity of $\OT_\mu$ we have Galerkin orthogonality and immediately derive 
Cea's Lemma and thus the best-approximation error estimate for $u\in H^s(0,T)$
\begin{equation*}
    \norm{u-u_h}_{H^1_\mu} \leq \inf_{w_h\in X_h}\norm{u-w_h}_{H^1_\mu}\leq c(h^{s-1} + \mu h^s)|u|_{H^{s}}, \,s\in [1,2].  
\end{equation*}    
As an illustrative example, we consider $T=1$, and the function 
\begin{equation*}
    u(t) = t^2(T-t)^{3/4}\in H^s(0,T),\, s<\frac{5}{4}
\end{equation*}
solving \eqref{eq:ODE-wave} for $f \in H^{-\sigma}(0,T)$, $\sigma>\tfrac{3}{4}$. The solutions for $\mu\in\{1,10^5\}$ are depicted in Figure \ref{fig:solutions}, and the convergence rates are listed in Table \ref{tab:error_eoc_mu_comparison}. We see, that the method is stable w.r.t. to $\mu$ and gives optimal orders of convergence in $H^1$ right from the start. 

\begin{figure}[htbp!]
    \centering
    \begin{subfigure}{0.45\textwidth}
        \centering
        \includegraphics[width=\textwidth]{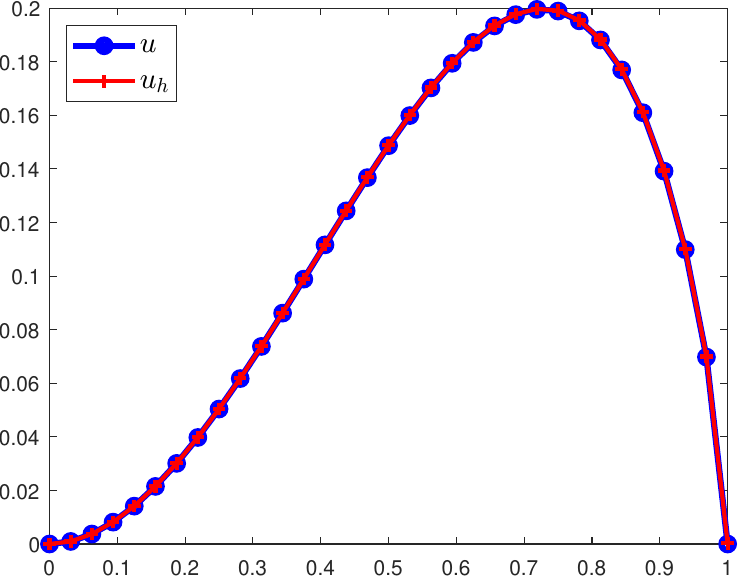}
        \caption{$\mu=1$}
    \end{subfigure}
    \begin{subfigure}{0.45\textwidth}
        \centering
        \includegraphics[width=\textwidth]{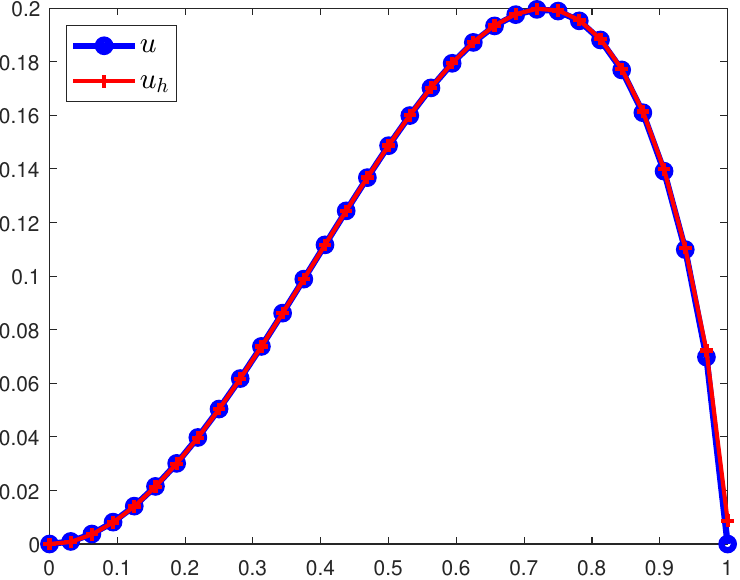}
        \caption{$\mu=10^5$}
    \end{subfigure}
    \caption{Exact solution $u$ and reconstruction $u_h$ for different $\mu$ on $N=32$ elements.}
    \label{fig:solutions}
\end{figure}

\begin{table}[ht]
\centering
\begin{tabular}{cc|cc|cc|cc|cc}
\toprule
& & \multicolumn{4}{c|}{$\mu=1$} & \multicolumn{2}{c|}{$\mu=1000$} & \multicolumn{2}{c}{$\mu=10^5$} \\
$N$ & $h$ 
& $\norm{u-u_h}_{L^2}$ & eoc 
& $|u-u_h|_{H^1}$ & eoc 
& $|u-u_h|_{H^1}$ & eoc 
& $|u-u_h|_{H^1}$ & eoc \\
\midrule
4   & 0.250 & 2.13e-02 & 0.00 & 3.19e-01 & 0.00 & 3.29e-01 & 0.00 & 3.30e-01 & 0.00 \\
8   & 0.125 & 7.70e-03 & 1.47 & 2.23e-01 & 0.52 & 2.26e-01 & 0.54 & 2.29e-01 & 0.53 \\
16  & 0.063 & 2.89e-03 & 1.41 & 1.62e-01 & 0.46 & 1.63e-01 & 0.47 & 1.67e-01 & 0.46 \\
32  & 0.031 & 1.13e-03 & 1.35 & 1.23e-01 & 0.39 & 1.24e-01 & 0.40 & 1.27e-01 & 0.39 \\
64  & 0.016 & 4.57e-04 & 1.31 & 9.80e-02 & 0.33 & 9.80e-02 & 0.34 & 1.00e-01 & 0.34 \\
128 & 0.008 & 1.88e-04 & 1.28 & 7.99e-02 & 0.30 & 7.98e-02 & 0.30 & 8.08e-02 & 0.31 \\
256 & 0.004 & 7.81e-05 & 1.27 & 6.60e-02 & 0.27 & 6.60e-02 & 0.27 & 6.62e-02 & 0.29 \\
512 & 0.002 & 3.26e-05 & 1.26 & 5.51e-02 & 0.26 & 5.51e-02 & 0.26 & 5.51e-02 & 0.28 \\
\bottomrule
\end{tabular}
\caption{Errors and order of convergence for different $\mu\in\{1,1000,10^5\}$.}
\label{tab:error_eoc_mu_comparison}
\end{table}
\vspace{-2em}

\section{Conclusions}
We proposed a novel transformation operator for an ODE that is related to a space-time 
FEM formulation of the wave equation, resulting in a Galerkin-Bubnov formulation that 
is unconditionally stable and coercive. We can theoretically prove stability and best approximation error estimates independent of $\mu$. This opens the door for a space-time transformation that leads to a coercive formulation. Related results will be published elsewhere. 

\section*{Acknowledgments} 
The research leading to this 
publication received funding from the Dutch Research 
Council (NWO) under the NWO-Talent Programme Veni with the project number 
VI.Veni.212.253.

\bibliographystyle{plain}

\begin{thebibliography}{.99}
  \bibitem{BignardiMoiola:2025} P. Bignardi, A. Moiola: \textit{A space-time continuous and coercive formulation for the wave equation}. arXiv: 2312.07268 (2025). 
  \bibitem{FerrariPerugiaZampa:2025} M. Ferrari, I. Perugia, E. Zampa: \textit{Inf-sup stable space-time discretization of the wave equation based on a first-order-in-time variational formulation.} arXiv:2506.05886 (2025)
  \bibitem{FraschiniLoliMoiolaSangalli:2024} S. Fraschini, G. Loli, A. Moiola, G. Sangalli:
    \textit{An unconditionally stable space–time isogeometric method for the acoustic wave equation,}
    Comput. Math. Appl.
    vol. 169,
    pp. 205--222 (2024)
  \bibitem{FuehrerGonzalesKarkulik:2025} T. F\"uhrer, R. Gonz\'{a}lez, M. Karkulik: \textit{Well-posedness of first-order acoustic wave equations and space-time finite element approximation.} IMA J. Numer. Anal., drae104 (2025)
  \bibitem{HLSU:2023} D.~Hoonhout, R.~L\"oscher, O.~Steinbach, C.~Urz\'ua-Torres: \textit{Stable least squares space time boundary element methods for the wave equation.} arXiv:2312.12547 (2023)
  \bibitem{KoetheLoescherSteinbach:2023} C. K\"othe, R. L\"oscher, O. Steinbach: \textit{Adaptive least-squares space-time finite element methods.} arXiv:2309.14300 (2023)
  \bibitem{SteinbachZank:2020a} O.Steinbach, M- Zank: \textit{A stabilized space-time finite element method for the wave equation.} In: Advanced finite element methods with applications. Lect. Notes Comput. Sci. Eng., vol. 128, pp. 341--370 (2019) 
  \bibitem{SteinbachZank:2020} O.~Steinbach, M.~Zank: \textit{Coercive space-time finite element methods for initial boundary value problems}, ETNA 52 (2020) 154--194.  
  \bibitem{SteinbachZank:2021a}
  O.~Steinbach, M.~Zank: \textit{A generalized inf-sup stable variational
  formulation for the wave equation}. J. Math. Anal. Appl. 505 (2022) 125457.
  \bibitem{Zank:2020} M.~Zank: \textit{Inf-sup Stable Space-Time Methods for Time-Dependent Partial Differential Equations}. PhD Thesis, TU Graz (2020).  
\end{thebibliography}

\end{document}